\documentclass[12pt, twoside]{article}
\usepackage{amsmath,amsthm,amssymb}
\usepackage{times}
\usepackage{enumerate}
\usepackage{mathrsfs}
\usepackage[all]{xy}
\usepackage{geometry}
\usepackage{graphicx}
\usepackage{bm}
\usepackage{microtype}
\usepackage[colorlinks=true,linkcolor=blue,citecolor=red,urlcolor=cyan,hidelinks]{hyperref}
\usepackage{hyperref}
\usepackage{cleveref}
\usepackage{tocbibind} 

\usepackage{titlesec}
\titleformat{\section}
  {\normalfont\fontsize{12}{14}\bfseries}{\thesection}{1em}{}
\titleformat{\subsection}
  {\normalfont\fontsize{12}{14}\bfseries}{\thesubsection}{1em}{}
\titleformat{\subsubsection}
  {\normalfont\fontsize{12}{14}\bfseries}{\thesubsubsection}{1em}{}

\def\titlerunning#1{\gdef\titrun{#1}}
\makeatletter
\def\author#1{\gdef\autrun{\def\and{\unskip, }#1}\gdef\@author{#1}}
\def\address#1{{\def\and{\\\hspace*{18pt}}\renewcommand{\thefootnote}{}%
\footnote {#1}}%
\markboth{\autrun}{\titrun}}
\makeatother
\def\email#1{e-mail: #1}

\def\keywords#1{\par\medskip
\noindent\textbf{Keywords.} #1}

\newtheorem{thm}{Theorem}[section]
\newtheorem{cor}[thm]{Corollary}

\newtheorem{prop}[thm]{Proposition}

\theoremstyle{definition}
\newtheorem{defin}[thm]{Definition}
\newtheorem{rem}[thm]{Remark}
\newtheorem{exa}[thm]{Example}

\numberwithin{equation}{section}

\crefname{thm}{Theorem}{Theorems}
\crefname{cor}{Corollary}{Corollaries}
\crefname{lem}{Lemma}{Lemmas}
\crefname{defin}{Definition}{Definitions}
\crefname{exa}{Example}{Examples}
\crefname{rem}{Remark}{Remarks}
\crefname{equation}{Equation}{Equations}

\begin{document}

\titlerunning{}

\title{On Kodaira dimension and scalar curvature in almost Hermitian geometry}

\author{Xianchao Zhou}

\date{\today}

\maketitle

\address{School of Mathematical Sciences, Zhejiang University of Technology, Hangzhou 310023, China\\
\email{zhouxianch07@zjut.edu.cn}}


\begin{abstract}
In this paper, we investigate Riemannian curvature constraints on the Kodaira dimension of compact almost Hermitian manifolds. Specifically, for a compact almost Hermitian manifold $(M, J, g)$ in the Gray-Hervella class $\mathcal{W}_2\oplus\mathcal{W}_3\oplus \mathcal{W}_4$ with nonnegative Riemannian scalar curvature, we prove that its Kodaira dimension must satisfy $\kappa(M, J)=-\infty$; or $\kappa(M, J)=0$, in which case $(M,J,g)$ is a K\"{a}hler  Calabi-Yau manifold. The same conclusions also hold for compact Hermitian manifolds with an assumption of nonnegative mixed scalar curvature. As an important example, we study the twistor geometry of a compact anti-self-dual 4-manifold. In particular, for the twistor space with the Eells-Salamon almost complex structure, we show that the Kodaira dimension is zero.
\end{abstract}

\keywords{Riemannian scalar curvature, $J$-scalar curvature, Kodaira dimension, Chern scalar curvature, Gauduchon metric}


\section{Introduction}

On a compact almost Hermitian manifold, several significant scalar curvatures arise, such as Riemannian scalar curvature, $J$-scalar curvature and Chern scalar curvature. The signs of these scalar curvatures or the corresponding total scalar curvatures are closely related to the underlying manifold's topology and complex algebraic geometry.

In K\"{a}hler geometry, the complex structure is parallel with respect to the Levi-Civita connection, then these scalar curvatures are identical within a factor of two. Hence on a K\"ahler manifold, the complex geometry is very compatible with the underlying Riemannian geometry. Yau \cite{Yau1} proved that the Kodaira dimension of a compact  K\"{a}hler manifold with positive total
scalar curvature must be $-\infty$. Furthermore, Yau \cite{Yau1} also proved that a
compact K\"{a}hler surface is uniruled if and only if it admits a K\"{a}hler metric with positive total scalar curvature. On the other
hand,  for compact K\"{a}hler surfaces, LeBrun \cite{LeB} obtained an interesting relation between the Kodaira dimension and the Yamabe invariant. Heier-Wong \cite{HeW1} studied the total scalar curvatures of K\"{a}hler metrics on projective manifold. In particular, Heier-Wong \cite{HeW1} showed that a projective manifold is uniruled if it admits a K\"{a}hler metric with positive total scalar curvature.

For a Hermitian non-K\"{a}hler manifold, the complex structure is not parallel with respect to the Levi-Civita connection. One usually chooses the Chern connection instead of the Levi-Civita connection and hence use the Chern scalar curvature. Gauduchon \cite{Gau2} proved that if a compact complex manifold admits a Gauduchon metric with positive total Chern scalar curvature, then its Kodaira dimension is $-\infty$. Chiose-R\u{a}sdeaconu-\c{S}uvaina \cite{CRS} successfully extended Yau's result to the non-K\"ahler case and showed that
a compact Moishezon manifold is uniruled if and only if it admits a Gauduchon metric with positive total Chern scalar curvature.
In recent years, Yang \cite{Yang2} systematically investigated the relations among the total Chern scalar curvatures of Gauduchon metrics, Kodaira dimension and the pseudo-effectiveness of canonical line bundles. Later, from the perspective of Riemannian geometry, Yang \cite{Yang3} studied the Kodaira dimension of a compact Hermitian manifold with nonnegative Riemannian scalar curvature.

Recently, Chen-Zhang \cite{ChZ1, ChZ2} have generalized the notion of Kodaira dimension to compact almost complex manifold using the growth rate of the pseudoholomorphic sections of the corresponding canonical complex line bundle. We refer readers to section 3 for details. Motivated by these significant developments, our first main result extend Yang's work \cite{Yang3} to a broader class of manifolds within the framework of almost Hermitian geometry.

\begin{thm}\textup{(See Theorem 4.1).}
Let $(M,J,g)$ be a compact almost Hermitian manifold of complex dimension $n\geq 3$. If
$(M,J,g)\in \mathcal{W}_2\oplus\mathcal{W}_3\oplus \mathcal{W}_4$ and its Riemannian scalar curvature $s\geq0$, then either

(1) $\kappa(M, J)=-\infty$; or

(2) $\kappa(M, J)=0$, and $(M,J,g)$ is a K\"{a}hler Calabi-Yau manifold.

\end{thm}
\noindent Here, the class $\mathcal{W}_2\oplus\mathcal{W}_3\oplus \mathcal{W}_4$, introduced by Gray-Hervella \cite{GrH}, constitutes a broad family of almost Hermitian structures that includes the Hermitian ones (see Section 2 for details). We remark that the proof of Theorem 1.1 also applies to almost Hermitian surfaces (see Corollary 4.2).

Our second main result requires a condition on the mixed scalar curvature. For this purpose, the concept of $J$-scalar curvature (also called $\ast$-scalar curvature) must be introduced. In the framework of non-K\"{a}hler geometry, the $J$-scalar curvature generally differs from the Riemannian scalar curvature. Its definition involves both the Levi-Civita connection and the almost complex structure $J$. The precise definition has been provided in Section 2.

Recall that a compact complex manifold is called Moishezon  if it is bimeromorphic to a projective manifold.  A compact complex manifold is called uniruled if it is covered by rational curves. Then one has
\begin{thm}\textup{(See Corollary 5.2).}
Let $(M,J)$ be a compact Moishezon manifold of complex dimension $n$. If $(M,J)$ admits a Hermitian metric $g$ such that its corresponding Riemannian scalar curvature $s$ and $J$-scalar curvature $s_J$ satisfing $3s_J+s>0$, then  $(M,J)$ is uniruled.
\end{thm}

As an important example, in the appendix, we provide a detailed analysis of various scalar curvatures on the twistor space $(\mathbf{Z},\mathbb{J}_{\pm})$ of an anti-self-dual 4-manifold, where the complex structure $\mathbb{J}_+$ was introduced by Atiyah-Hitchin-Singer \cite{AHS} and the almost complex structure $\mathbb{J}_-$ was introduced by Eells-Salamon \cite{ES}. Building upon Hitchin's famous classification of K\"{a}hlerian twistor spaces
\cite {Hit} and Verbitsky's results of pluriclosed (or strong K\"{a}hler with torsion) metrics on twistor spaces \cite{Ver}, we obtain
\begin{thm}
Let $(N,g_N)$ be a compact anti-self-dual 4-manifold with metric $g_N$,  and $(\mathbf{Z},\mathbb{J}_{\pm})$ is the corresponding twistor space endowed with structures $\mathbb{J}_+$ and $\mathbb{J}_-$.

(1) If $(N,g_N)$ is also an Einstein manifold with negative scalar curvature, then
 $(\mathbf{Z}, \mathbb{J}_+)$ does not admit any pluriclosed metric, but admits balanced metrics with positive, negative and zero Riemannian scalar curvature (Chern scalar curvature, respectively). Furthermore, the canonical line bundle and the anti-canonical line bundle of $(\mathbf{Z}, \mathbb{J}_+)$ are not pseudo-effective.

(2) With respect to $\mathbb{J}_-$, then the Kodaira dimension $\kappa(\mathbf{Z},\mathbb{J}_-)=0$.
\end{thm}

The rest of this paper is organized as follows. In section 2, we provide a brief description of almost Hermitian geometry, including the
Gray-Hervella's classification of almost Hermitian structures, the relationship between the $J$-scalar curvature and the Riemannian scalar curvature. In section 3, we introduce a natural 2-parameter family of conformal invariants for Chern scalar curvatures.
We also recall the definition of Kodaira dimension for compact almost complex manifolds. In section 4, we present the proof of Theorem 1.1, and we also give some examples from twistor spaces. In section 5, we study the Kodaira dimension of compact Hermitian manifolds with an assumption of nonnegative mixed scalar curvature and offer some related remarks. Finally, in the appendix, a detailed analysis of various scalar curvatures on the twistor space is provided, which directly leads to the proof of Theorem 1.3.

In this paper, we will introduce the following notational conventions as in \cite{Gau4, FuZ}: the metric $g$ induces natural inner product $\langle\ , \ \rangle$ and the corresponding norm $|~|$ on the bundle $\Lambda^l M$ of all $l$-forms and  on the bundle $TM\otimes \Lambda^l M$ of all $TM$-valued $l$-forms. Quantities induced by other metrics will be indicated with subscripts.

\section{Preliminaries on almost Hermitian geometry}

Let $(M,J,g)$ be an almost Hermitian manifold of complex dimension $n$, $n\geq2$. $J$ is an almost complex structure which is compatible with
the Riemannian metric $g=\langle\ , \ \rangle$.
The fundamental 2-form $F$ associated to $g$ and $J$ is defined by
\[
F(X,Y)=g(JX,Y)
\]
for all vector fields $X,\,Y$ and the volume form of the metric $g$ is $dV_g=\frac{F^n}{n!}$.

Another important differential form of $(M,J,g)$ is the Lee form $\alpha_F=J\delta^g F$ where $\delta^g$ is the codifferential operator, i.e., the formal adjoint of $d$ in the metric $g$. It can be determined by the formula
\begin{equation}
dF=(dF)_0+\frac{1}{n-1}\alpha_F\wedge F
\end{equation}
where $(dF)_0$ is the primitive part of $dF$.
Denote by $(dF)^+$ the sum of $(2,1)$ and $(1,2)$ components of $dF$  and by $(dF)^-$ the sum of $(3,0)$ and $(0,3)$ components of $dF$.
The formula (2.1)  implies
\begin{equation}
(dF)^+=(dF)_0^+ +\frac{1}{n-1}\alpha_F\wedge F
\end{equation}
where $(dF)^+_0$ is the primitive part of $(dF)^+$.

The Nijenhuis tensor $N$ of $J$ is a $(1,2)$-type tensor defined by
\[
N(X,Y)=[X,Y]+J[JX,Y]+J[X,JY]-[JX,JY]
\]
for all vector fields $X,Y$. It is well-known that $N\equiv 0$ if and only if $J$ is integrable.
 $N$ can be viewed as a $(0,3)$-tensor:  $N(X,Y,Z)=\langle X, N(Y,Z)\rangle$.
Denote by $\mathfrak{b}N$ its skew-symmetric part:
\begin{equation*}
\mathfrak{b}N(X,Y,Z)=\frac{1}{3}\left(N(X,Y,Z)+N(Y,Z,X)+N(Z,X,Y)\right).
\end{equation*}
Define $N^0=N-\mathfrak b N$. Clearly, $\mathfrak{b}N^0=0$.

Denote by $\nabla$ the Levi-Civita connection of the metric $g$. We have
\begin{equation*}
(\nabla_X F)(Y,Z)=\langle(\nabla_X J)Y , Z \rangle.
\end{equation*}
We also use the notation:$(\nabla F)(X,Y,Z)=(\nabla_X F)(Y,Z)$.
According to Proposition 1 in \cite{Gau4}, $\nabla F$ can be decomposed as
\begin{equation}
\begin{aligned}
(\nabla_X F)&(Y,Z)
=\frac{1}{3}(dF)^-(X,Y,Z)-\frac{1}{2}N^0(JX,Y,Z)\\
&\quad \quad\quad+\frac{1}{2}(dF)^+(X,Y,Z)-\frac 1 2 (dF)^+(X,JY,JZ)
\end{aligned}
\end{equation}
for all vector fields $X,\,Y,\,Z$.

Through a direct calculation as in \cite{Gau4, FuZ}, one gets the following important pointwise formula
\begin{equation}
|\nabla F|^2=\frac{|\alpha_F|^2}{n-1}+|(dF)_0^+|^2+\frac{1}{4}|N^0|^2+\frac{1}{3}|(dF)^-|^2.
\end{equation}
For an almost Hermitian manifold, the above four components $(dF)^-$, $N^0$, $(dF)_0^+$ and $\alpha_F$ carry important geometric information. Gray-Hervella \cite{GrH} introduced that in a natural way there are precisely sixteen (four, respectively) classes of almost Hermitian manifolds of complex dimension $n\geq 3$ ($n=2$, respectively). In fact, the defining conditions of each class are by letting some elements of set
$\{(dF)^-, N^0, (dF)_0^+, \alpha_F\}$ equal to zero. For example, $\mathcal W_3\oplus\mathcal W_4$ is the class of Hermitian manifolds: $(dF)^-=N^0=0$; $\mathcal{W}_2\oplus\mathcal{W}_3\oplus \mathcal{W}_4$ is the class: $(dF)^-=0$.

The Riemannian curvature tensor $R$ of the metric $g$ is defined by
\begin{equation*}
R(X,Y,Z,W)=\langle\nabla_Z \nabla_W Y-\nabla_W \nabla_Z Y-\nabla_{[Z,W]}Y, X\rangle
\end{equation*}
for all vector fields $X,\,Y,\,Z,\,W$.

Choose a local $J$-adapted $g$-orthonormal frame field
$\{e_i, e_{n+i}=Je_i\}_{i=1,2,\dots,n}$ and hence a unitary frame field
$\{u_i=\frac{1}{\sqrt{2}}(e_i-\sqrt{-1}e_{n+i})\}_{i=1,2,\dots,n}$. Set $u_{\bar{i}}=\overline{u_i}$. Then the Ricci tensor ${Ric}$ of $g$ is defined by
$Ric(X,Y)=\sum_{A=1}^{2n} R(e_A,X,e_A,Y)$ and the Riemannian scalar curvature $s$ of $g$ is defined by $s=\sum_{A=1}^{2n} {Ric}(e_A,e_A)$.

On the other hand, on an almost Hermitian manifold $(M,J,g)$, there is a $J$-twisted version of the Ricci tensor called the $J$-Ricci tensor
(also called $\ast$-Ricci tensor in some literatures) \cite{GrH,TrV,dRS} and denoted by $Ric_{J}$. It is defined by
\begin{equation*}
Ric_{J}(X,Y)=\sum_{A=1}^{2n} R(e_A,X,Je_A,JY).
\end{equation*}
In general,
$Ric_J\neq Ric$, $Ric_J (X,Y)\neq Ric_J(Y,X)$. However, we have $Ric_J (X,Y)=Ric_J(JY,JX)$.
We define the $J$-scalar curvature $s_J$ (also called $\ast$-scalar curvature) as
\begin{equation*}
s_J=\sum_{A=1}^{2n} Ric_J (e_A,e_A).
\end{equation*}
Here, we adopt the notations ($Ric_J$ and $s_J$) consistent with \cite{dRS} to emphasize the direct connection between these geometric quantities and almost complex structure $J$, while avoiding confusion with Hodge $\ast_g$-operator.

With respect to a unitary frame $\{u_i\}_{i=1,2,\dots,n}$, the $J$-scalar curvature $s_J$  and  the Riemannian scalar curvature $s$ of $g$ are also given by
\begin{equation}
s_J=2\sum_{i,j=1}^{n}R(u_{\bar{i}},u_i,u_j,u_{\bar{j}}),
\end{equation}
\begin{equation}
\begin{aligned}
s&=4\sum_{i,j=1}^{n}R(u_{\bar{i}},u_{\bar{j}},u_i,u_j)+2\sum_{i,j=1}^{n}R(u_{\bar{i}},u_i,u_j,u_{\bar{j}})\\
&=4\sum_{i,j=1}^{n}R(u_{\bar{i}},u_j,u_i,u_{\bar{j}})-2\sum_{i,j=1}^{n}R(u_{\bar{i}},u_i,u_j,u_{\bar{j}}).
\end{aligned}
\end{equation}

Moreover,  from \cite[(2.11)]{FuZ}, on an almost Hermitian manifold, we have the following useful relation,
\begin{equation}
s_J=s-\frac{2}{3}|(dF)^-|^2+\frac{1}{4}|N^0|^2-|\alpha_F|^2-2\delta^g\alpha_F.
\end{equation}

\section{Chern scalar curvatures and Kodaira dimension}

As in section 2, let $(M,J,g)$ be an almost Hermitian manifold and $\nabla$ the Levi-Civita connection of the metric $g$.
The Chern connection $D$ is the unique connection characterized by the conditions $D J=0$, $D g=0$, and the vanishing of the (1,1)-component of its torsion. The relationship between the Chern connection $D$ and the Levi-Civita connection $\nabla$ is given by\cite{Gau4}
\begin{equation*}
\begin{aligned}
&g(D_X Y,Z)
=g(\nabla_X Y-\frac{1}{2}J(\nabla_X J)Y,Z)\\
&\quad\quad\quad\quad\quad
+\frac{1}{4}g\left((\nabla_{JY}J) Z+J(\nabla_Y J)Z,X\right)-\frac{1}{4}g\left((\nabla_{JZ}J) Y+J(\nabla_Z J)Y,X\right).
\end{aligned}
\end{equation*}
The Chern curvature tensor $R^{\mathrm{Ch}}$ is defined by
\begin{equation*}
R^{\mathrm{Ch}}(X,Y,Z,W)=\langle D_Z D_W Y-D_W D_Z Y-D_{[Z,W]} Y, X\rangle
\end{equation*}
for all  vector fields $X,Y,Z,W$.

With respect to a unitary frame $\{u_i\}_{i=1,2,\dots,n}$, the Chern-Ricci form is
\begin{equation}
\rho_g=\sqrt{-1}\sum_{i=1}^{n}R^{\mathrm{Ch}}(u_{\bar{i}},u_i, ., .)\in 2\pi c_1(M, J),
\end{equation}
where $c_1(M, J)$ is the first Chern class of $(M, J)$. Two scalar curvatures $S_1^{\mathrm{Ch}}$ and $S_2^{\mathrm{Ch}}$ of the Chern connection are defined, respectively, by
\begin{equation*}
S_1^{\mathrm{Ch}}=\sum_{i,j=1}^{n}R^{\mathrm{Ch}}(u_{\bar{i}},u_i,u_j,u_{\bar{j}})\quad \textup{and}\quad S_2^{\mathrm{Ch}}= \sum_{i,j=1}^{n}R^{\mathrm{Ch}}(u_{\bar{i}},u_j,u_i,u_{\bar{j}}).
\end{equation*}
It is worth noting that $S_1^{\mathrm{Ch}}$ is the usual Chern scalar curvature.

Moreover, from \cite[Theorem 1.1]{FuZ} or \cite{LeU} in a different version, we have the following important relations between
Chern scalar curvatures and Riemannian scalar curvature,

\begin{equation}
S_1^{\mathrm{Ch}}
=\frac{s}{2}-\frac{5}{12}|(dF)^-|^2+\frac{1}{16}|N^0|^2+\frac{1}{4}|(dF)_0^+|^2+\frac{1}{4(n-1)}
|\alpha_F|^2-\frac{1}{2}\delta^g \alpha_F
\end{equation}
and
\begin{equation}
S_2^{\mathrm{Ch}}
=\frac{s}{2}-\frac{1}{12}|(dF)^-|^2+\frac{1}{32}|N^0|^2+\frac{1}{4}|(dF)_0^+|^2
+\Bigl(\frac{1}{4(n-1)}-\frac{1}{2}\Bigr)|\alpha_F|^2-\delta^g \alpha_F.
\end{equation}

With respect to the Chern connection $D$, we define Chern laplacian  $\Delta_{g}^{\mathrm{Ch}}$ by
\begin{equation}
\Delta_{g}^{\mathrm{Ch}} f=-2\sum_{i=1}^{n} (Dd f)(u_i,u_{\bar{i}}),
\end{equation}
for any smooth function $f$ on $M$. There are other alternative ways to write $\Delta_g^{\mathrm{Ch}} f$ as\cite{Gau3,TWY}
\begin{equation}
\Delta_g^{\mathrm{Ch}} f=\Delta_g f+\langle \alpha_{F}, df \rangle=-\langle dJdf, F \rangle,
\end{equation}
where $\Delta_g=d \delta^g+\delta^g d$ is the Hodge laplacian.

We recall the foundational theorem by Gauduchon on the existence of a standard metric in any Hermitian conformal class.
\begin{thm}\cite{Gau1}
Let $(M,J,g)$ be a compact almost Hermitian manifold of complex dimension $n\geq 2$. Then  the conformal class $[g]$ admits a natural base-point $\hat{g}$, which is characterized by having a co-closed Lee form, after normalizing $\hat{g}$ to unit volume.
\end{thm}

We use the superscript ~$\hat{}$~ to denote quantities related to the Gauduchon metric $\hat{g}$. Based on this result by Gauduchon, we introduce the following mixed conformal invariant:
\begin{equation}
\Gamma_{\lambda,\mu}(M,J,[g])=\int_M \left(\lambda \hat{S}_1^{\mathrm{Ch}}+\mu \hat{S}_2^{\mathrm{Ch}}\right)dV_{\hat{g}},
\end{equation}
where $\lambda$ and $\mu$ are two constants. In particular, $\Gamma_{1,0}(M,J,[g])$ corresponds to the degree of the
anti-canonical line bundle $K_M^{-1}$.

The following proposition extends a result of Balas \cite[Proposition 1.8]{Bal} to the almost Hermitian setting.
\begin{prop}
Let $(M,J,g)$ be a compact almost Hermitian manifold of complex dimension $n$. If $n\lambda + \mu \ne 0$, then
there exists a conformal metric $\tilde{g} \in [g]$ whose mixed Chern scalar curvature $\lambda \tilde{S}_1^{\mathrm{Ch}} + \mu \tilde{S}_2^{\mathrm{Ch}}$
has the same sign as $\Gamma_{\lambda,\mu}(M,J,[g])$.
\end{prop}

\begin{proof}
 From Theorem 3.1, for convenience, we can assume that the background metric $g$ is just the Gauduchon metric $\hat{g}$ with unit volume. By direct calculations, the adjoint
of the Chern laplacian $\Delta_{\hat{g}}^{\mathrm{Ch}}$ on smooth function $f$ is
\begin{equation}
(\Delta_{\hat{g}}^{\mathrm{Ch}})^{\ast} f=\Delta_{\hat{g}} f-\langle \alpha_{\hat{F}}, df \rangle_{\hat{g}},
\end{equation}
where $\alpha_{\hat{F}}$ denotes the Lee form of the Gauduchon metric $\hat{F}$.

Notice that $\mathrm{Ker}(\Delta_{\hat{g}}^{\mathrm{Ch}})^{\ast}$ are the constant functions. Indeed, take $f\in \mathrm{Ker}(\Delta_{\hat{g}}^{\mathrm{Ch}})^{\ast}$, one has
\begin{equation}
0=\int_M f (\Delta_{\hat{g}}^{\mathrm{Ch}})^{\ast} f dV_{\hat{g}}=\int_M \left(|\nabla f|_{\hat{g}}^2
-\frac{1}{2}\langle d (f^2), \alpha_{\hat{F}}\rangle_{\hat{g}}\right)dV_{\hat{g}}=\int_M |\nabla f|_{\hat{g}}^2 dV_{\hat{g}}
\end{equation}
since $\delta^{\hat{g}}\alpha_{\hat{F}}=0$.

For $\Delta_{\hat{g}}^{\mathrm{Ch}}$
and $(\Delta_{\hat{g}}^{\mathrm{Ch}})^{\ast}$ are elliptic operators on a compact manifold, we have the following orthogonal decomposition of
the space of smooth functions,
\begin{equation}
C^{\infty}(M)=\mathrm{Ker}(\Delta_{\hat{g}}^{\mathrm{Ch}})^{\ast}\oplus \mathrm{Im} \Delta_{\hat{g}}^{\mathrm{Ch}}.
\end{equation}
Hence  if $n\lambda + \mu \ne 0$, the equation
\begin{equation}
(n\lambda + \mu) \Delta_{\hat{g}}^{\mathrm{Ch}} f =\Gamma_{\lambda,~\mu}(M,J,[g])-
(\lambda \hat{S}_1^{\mathrm{Ch}}+\mu \hat{S}_2^{\mathrm{Ch}})
\end{equation}
is solvable for $f$, since the right hand side is orthogonal to the constant. Now, define $\tilde{g}=e^{2f}\hat{g}$,  then the
conformal variations of two Chern scalar curvatures are\cite{LeU, LZZ}
\begin{equation}
e^{2f}\tilde{S}_1^{\mathrm{Ch}}=\hat{S}_1^{\mathrm{Ch}}+n \Delta_{\hat{g}}^{\mathrm{Ch}}f
\end{equation}
and
\begin{equation}
e^{2f}\tilde{S}_2^{\mathrm{Ch}}=\hat{S}_2^{\mathrm{Ch}}+ \Delta_{\hat{g}}^{\mathrm{Ch}}f.
\end{equation}
From (3.11),(3.12) and (3.10), we have
\begin{equation}
\begin{aligned}
\lambda \tilde{S}_1^{\mathrm{Ch}}+\mu \tilde{S}_2^{\mathrm{Ch}}
&=e^{-2f}\left(\lambda \hat{S}_1^{\mathrm{Ch}}+\mu \hat{S}_2^{\mathrm{Ch}}+(n\lambda + \mu)\Delta_{\hat{g}}^{\mathrm{Ch}}f\right)\\
&=e^{-2f}\Gamma_{\lambda,~\mu}(M,J,[g]),
\end{aligned}
\end{equation}
i.e., the mixed Chern scalar curvature $\lambda \tilde{S}_1^{\mathrm{Ch}} + \mu \tilde{S}_2^{\mathrm{Ch}}$ of the metric $\tilde{g}=e^{2f}\hat{g}$ has the same sign as $\Gamma_{\lambda,\mu}(M,J,[g])$.
\end{proof}

\begin{rem}
If $\lambda=1,~\mu=0$,  this case is due to Lejmi-Upmeier\cite[Proposition 5.7]{LeU}. On the other hand, if $n\lambda + \mu =0$, one has
\begin{equation*}
\lambda \tilde{S}_1^{\mathrm{Ch}}+\mu \tilde{S}_2^{\mathrm{Ch}}=e^{-2f}(\lambda \hat{S}_1^{\mathrm{Ch}}+\mu \hat{S}_2^{\mathrm{Ch}}),
\end{equation*}
i.e., $\lambda \tilde{S}_1^{\mathrm{Ch}}+\mu \tilde{S}_2^{\mathrm{Ch}}$ has the same sign for every metric in the conformal class $[g]$.
\end{rem}

In the last part of this section, following Chen-Zhang\cite{ChZ1}, we recall briefly the definition of Kodaira dimension of  a
compact almost complex manifold $(M,J)$. Let $\Omega^{p,q}(M)=\Gamma(M, \Lambda^{p,q}M)$ be the space of $(p,q)$-forms, then the exterior differential $d$
splits as
\begin{equation*}
d \Omega^{p,q}(M)\subset \Omega^{p+2,q-1}(M)+\Omega^{p+1,q}(M)+\Omega^{p,q+1}(M)+\Omega^{p-1,q+2}(M),
\end{equation*}
i.e., $d=N_J+\partial+\bar{\partial}+\bar{N}_J$, where $N_J=\pi^{p+2,q-1}\circ d$, $\bar{\partial}=\pi^{p,q+1}\circ d$, and the notation $\pi^{p,q}$ is the projection onto the $(p,q)$-forms.

Let $K_M=\Lambda^{n,0}M$ be the canonical line bundle, then the $\bar{\partial}$-operator gives rise to a pseudoholomorphic structure
on $K_M$, i.e., a differential operator still denoted by $\bar{\partial}$,
\begin{equation*}
\bar{\partial}: \Gamma(M, K_M)\longrightarrow \Gamma(M, \Lambda^{0,1}M \otimes K_M),
\end{equation*}
satisfying the Leibniz rule
\begin{equation*}
\bar{\partial}(f\sigma)=\bar{\partial}f\otimes \sigma+f\bar{\partial}\sigma
\end{equation*}
for every smooth function $f$ and section $\sigma$.

Inductively, we can extend the $\bar{\partial}$-operator to get a natural pseudoholomorphic structure (denoted by $\bar{\partial}_m$) on the
pluricanonical bundle $K_M^{\otimes m}$, for any  $m\geq2$. Meanwhile, for any $m$, the Hermitian metric $g$ induces  a Hermitian structure $h$ on the complex line bundle $K_M^{\otimes m}$. It should also be noted that the Chern connection $D$ of $(M,J,g)$ induces the unique Hermitian connection on  the Hermitian line bundle $(K_M^{\otimes m}, \bar{\partial}_m, h)$.

The space of pseudoholomorphic sections of $K_M^{\otimes m}$ is defined as
\begin{equation*}
H^0(M, K_M^{\otimes m})=\{\sigma\in\Gamma(M, K_M^{\otimes m}) \mid \bar{\partial}_m \sigma=0\}.
\end{equation*}
For a compact almost complex manifold $(M,J)$, using  Hodge theory,  $H^0(M, K_M^{\otimes m})$ is a finite dimensional complex vector space for every $m \geq 1$.
\begin{defin}
The $m$\textsuperscript{th}-plurigenus of $(M,J)$ is defined by
\begin{equation*}
P_m(M,J)=\dim_{\mathbb{C}} H^0(M, K_M^{\otimes m}).
\end{equation*}
The Kodaira dimension of $(M,J)$ is defined by
\begin{equation*}
\kappa(M, J)=
\begin{cases}
-\infty, & \text{if } P_m(M,J) = 0 \text{ for every } m \geq 1, \\
\limsup\limits_{m \to +\infty} \dfrac{\log P_m(M,J)}{\log m}, & \text{otherwise}.
\end{cases}
\end{equation*}
\end{defin}

If the almost complex structure $J$ is integrable, then $\bar{\partial}_m\circ\bar{\partial}_m=0$, the above definition is
the same as the famous Kodaira dimension of a complex manifold. Determining the Kodaira dimension of an almost complex manifold is an interesting yet extremely challenging problem. We refer readers to \cite{ChZ1, CNT} for computations on specific examples.

\section{The case of nonnegative Riemannian scalar curvature}

\begin{thm}
Let $(M,J,g)$ be a compact almost Hermitian manifold of complex dimension $n\geq 3$. If
$(M,J,g)\in \mathcal{W}_2\oplus\mathcal{W}_3\oplus \mathcal{W}_4$ and its Riemannian scalar curvature $s\geq0$, then either

(1) $\kappa(M, J)=-\infty$; or

(2) $\kappa(M, J)=0$, and $(M,J,g)$ is a K\"{a}hler Calabi-Yau manifold.

\end{thm}

\begin{proof}
We need to determine the sign of the total Chern scalar curvature of the Gauduchon metric $\hat{g}=e^{2f} g$ in the conformal class
$[g]$. From the formula of conformal deformation (3.11), one has
\begin{equation}\label{f1}
\begin{aligned}
\int_M \hat{S}_1^{\mathrm{Ch}}dV_{\hat{g}}
&=\int_M  e^{-2f}\bigl(S_1^{\mathrm{Ch}}+n\Delta_g^{\mathrm{Ch}}f\bigr)dV_{\hat{g}}\\
&=\int_M e^{-2f}S_1^{\mathrm{Ch}}dV_{\hat{g}}.
\end{aligned}
\end{equation}
The last equality is from the following two facts: the relation of the Chern Laplacians  for the metrics $g$ and $\hat{g}$,
\begin{equation}
\Delta_{\hat{g}}^{\mathrm{Ch}} v=e^{-2f}\Delta_{g}^{\mathrm{Ch}}v;
\end{equation}
with respect to the Gauduchon metric $\hat{g}$, the integral
\begin{equation}
\int_M \Delta_{\hat{g}}^{\mathrm{Ch}} v dV_{\hat{g}}=0
\end{equation}
for any smooth function $v$ on $M$.

Together with the formula (3.2), (\ref{f1}) can be refined as follows,
\begin{equation}\label{f2}
\begin{aligned}
\int_M \hat{S}_1^{\mathrm{Ch}}dV_{\hat{g}}
&=\int_M  e^{-2f}\left(\frac{1}{2}s-\frac{5}{12}|(dF)^-|^2+\frac{1}{16}|N^0|^2+\frac{1}{4}|(dF)^+|^2
-\frac{1}{2}\delta^g \alpha_F \right)dV_{\hat{g}}\\
&=\int_M  e^{-2f}\left(\frac{1}{2}s+\frac{1}{16}|N^0|^2+\frac{1}{4}|(dF)^+|^2
-\frac{1}{2}\delta^g \alpha_F \right)dV_{\hat{g}}
\end{aligned}
\end{equation}
for $(M,J,g)\in \mathcal{W}_2\oplus\mathcal{W}_3\oplus \mathcal{W}_4$: $(dF)^-=0$.

For the conformal metric $\hat{g}=e^{2f} g$, the corresponding Lee forms $\alpha_F$ and $\alpha_{\hat{F}}$ are related by
\begin{equation*}
\alpha_{\hat{F}}=\alpha_F+2(n-1)df.
\end{equation*}
From Besse \cite[Theorem 1.159]{Bes}, applying the codifferentials $\delta^{\hat{g}}$ and $\delta^g$ to the Lee form $\alpha_F$ yields

\begin{equation}
\delta^{\hat{g}}\alpha_F=e^{-2f} \left(\delta^g \alpha_F-(2n-2)\langle df, \alpha_F \rangle \right).
\end{equation}
Then
\begin{equation}\label{f3}
\begin{aligned}
\int_M  e^{-2f} \delta^g \alpha_F dV_{\hat{g}}
&=\int_M  \left(\delta^{\hat{g}}\alpha_F+e^{-2f}(2n-2)\langle df, \alpha_F \rangle \right)dV_{\hat{g}}\\
&=(2n-2)\int_M \langle df, \alpha_{\hat{F}}-(2n-2)df \rangle_{\hat{g}}dV_{\hat{g}}\\
&=-4(n-1)^2\int_M \langle df,df\rangle_{\hat{g}}dV_{\hat{g}}.
\end{aligned}
\end{equation}

From (\ref{f2}) and (\ref{f3}), we have
\begin{equation}\label{f4}
\begin{aligned}
\int_M \hat{S}_1^{\mathrm{Ch}}dV_{\hat{g}}
=&\int_M  e^{-2f}\left(\frac{1}{2}s+\frac{1}{16}|N^0|^2+\frac{1}{4}|(dF)^+|^2\right)dV_{\hat{g}}\\
&+2(n-1)^2\int_M \langle df,df\rangle_{\hat{g}}dV_{\hat{g}}\geq 0
\end{aligned}
\end{equation}
if the Riemannian scalar curvature $s\geq0$.

If $\int_M \hat{S}_1^{\mathrm{Ch}}dV_{\hat{g}}>0$, using the case: $\lambda=1, \mu=0$ of Proposition 3.2, there exists a metric
$\tilde{g}$ in the conformal class $[g]$ such that its Chern scalar curvature $\tilde{S}_1^{\mathrm{Ch}}>0$. Now, for any pseudoholomorphic section $\sigma$ of $K_M^{\otimes m}$, by the standard Bochner formula with respect to the metric $\tilde{g}$, one has
\begin{equation}\label{f5}
-\Delta_{\tilde{g}}^{\mathrm{Ch}} |\sigma|_{\tilde{g}}^2=2m\tilde{S}_1^{\mathrm{Ch}}|\sigma|_{\tilde{g}}^2+2|\tilde{D}\sigma|_{\tilde{g}}^2.
\end{equation}
Then from the maximum principle, $K_M^{\otimes m}$ has no nontrivial  pseudoholomorphic sections, i.e.,
the $m$\textsuperscript{th}-plurigenus
$P_m(M,J)=0$. Furthermore, the Kodaira dimension $\kappa(M, J)=-\infty$.

If $\int_M \hat{S}_1^{\mathrm{Ch}}dV_{\hat{g}}=0$, from (\ref{f4}), then the Riemannian scalar curvature $s=0$, $N^0=0$ and $(dF)^+=0$. Equivalently, $(M,J,g)$ is a K\"{a}hler manifold with zero Riemannian scalar curvature. Using formula (\ref{f5}) again, one has $P_m(M,J)\leq 1$ for any $m$. From Corollary 1.6 in \cite{Yang2} or Theorem 1.2 in \cite{Yang3}, we have  $\kappa(M, J)=-\infty$, or  $\kappa(M, J)=0$.
Moreover, when $\kappa(M, J)=0$, $(M,J)$ admits a K\"{a}hler Calabi-Yau structure. By the celebrated Calabi-Yau theorem \cite{Yau2}, there exists a K\"{a}hler Ricci-flat metric $g_{CY}$, i.e., $\rho_{g_{CY}}=0$.  The Chern-Ricci forms satisfy
\begin{equation}
\rho_{g} = \rho_{g} - \rho_{g_{CY}} =\sqrt{-1}\partial\bar{\partial}\log \left(\frac{dV_{g_{CY}}}{dV_{g}}\right).
\end{equation}
Since the metric $g$ has zero Riemannian scalar curvature, then
\begin{equation}
-\Delta_g^{\mathrm{Ch}}\log \left(\frac{dV_{g_{CY}}}{dV_{g}}\right)
=2\langle\sqrt{-1}\partial\bar{\partial}\log \left(\frac{dV_{g_{CY}}}{dV_{g}}\right),F\rangle=2\langle\rho_{g}, F\rangle=0,
\end{equation}
which implies $\log \left(\frac{dV_{g_{CY}}}{dV_{g}}\right)$ is a constant, thus $\rho_{g}=0$. This completes the proof.
\end{proof}

For almost Hermitian surface, the components of $\nabla F$ satisfy $(dF)^-=(dF)_0^+=0$ and $N^0=N$. Then as in the proof of the above Theorem, we have
\begin{cor}
Let $(M,J,g)$ be a compact almost Hermitian surface with nonnegative Riemannian scalar curvature. Then
either

(1) $\kappa(M, J)=-\infty$; or

(2) $\kappa(M, J)=0$, and $(M,J,g)$ is a K\"{a}hler Calabi-Yau surface.

\end{cor}

\begin{exa}
The twistor construction provides a large class of non-K\"{a}hler manifolds of complex dimension 3, building upon Hitchin's classification of K\"{a}hler twistor spaces \cite{Hit} and Taubes' fundamental results on the existence of anti-self-dual conformal structures \cite{Tau}.
The twistor space $\mathbf{Z}$ of a compact anti-self-dual 4-manifold admits both the Atiyah-Hitchin-Singer complex structure $\mathbb{J}_+$ \cite{AHS} and the Eells-Salamon almost complex structure $\mathbb{J}_-$ \cite{ES}. It is well-known that $(\mathbf{Z},\mathbb{J}_+)$ is uniruled and has Kodaira dimension $\kappa(\mathbf{Z},\mathbb{J}_+)=-\infty$. For the almost complex structure $\mathbb{J}_-$, we prove that the Kodaira dimension $\kappa(\mathbf{Z},\mathbb{J}_-)=0$ in the appendix. Furthermore, for a compact anti-self-dual  Einstein 4-manifold with negative scalar curvature, the corresponding twistor space $(\mathbf{Z}, \mathbb{J}_{\pm})$ is non-K\"{a}hler, but admits balanced metrics with positive, negative and zero Riemannian scalar curvature, respectively.
\end{exa}

\section{The case of a nonnegative mixed scalar curvature}

\begin{thm}
Let $(M,J,g)$ be a compact Hermitian manifold of complex dimension $n$. If its Riemannian scalar curvature $s$ and $J$-scalar
curvature $s_J$ satisfy $3s_J+s\geq 0$, then either

(1) $\kappa(M, J)=-\infty$; or

(2) $\kappa(M, J)=0$, and $(M,J,g)$ is a K\"{a}hler Calabi-Yau manifold.

\end{thm}

\begin{proof}
Let $\hat{g}=e^{2f} g$ be the unique Gauduchon metric with unit volume in the conformal class $[g]$, then
\begin{equation}\label{f6}
\int_M \hat{S}_1^{\mathrm{Ch}}dV_{\hat{g}}=\frac{1}{2}\int_M (\hat{S}_1^{\mathrm{Ch}}-\hat{S}_2^{\mathrm{Ch}}) dV_{\hat{g}}
+\frac{1}{2}\int_M (\hat{S}_1^{\mathrm{Ch}}+\hat{S}_2^{\mathrm{Ch}}) dV_{\hat{g}}.
\end{equation}
From Gauduchon \cite{Gau3}, we have
\begin{equation}
\int_M (\hat{S}_1^{\mathrm{Ch}}-\hat{S}_2^{\mathrm{Ch}}) dV_{\hat{g}}=\frac{1}{2}\int_M |\alpha_{\hat{F}}|_{\hat{g}}^2 dV_{\hat{g}}.
\end{equation}

On the other hand, using (3.11) and (3.12), we have

\begin{equation}
\begin{aligned}
\int_M (\hat{S}_1^{\mathrm{Ch}}+\hat{S}_2^{\mathrm{Ch}}) dV_{\hat{g}}
&=\int_M  e^{-2f}\left(S_1^{\mathrm{Ch}}+S_2^{\mathrm{Ch}}+(n+1)\Delta_{g}^{\mathrm{Ch}}f \right)dV_{\hat{g}}\\
&=\int_M  e^{-2f}(S_1^{\mathrm{Ch}}+S_2^{\mathrm{Ch}})dV_{\hat{g}}.
\end{aligned}
\end{equation}

For a Hermitian manifold, $(dF)^-=0$ and $N^0=0$, then (3.2) and (3.3) yield the following
relations between Chern scalar curvatures and Riemannian scalar curvature:
\begin{equation}
S_1^{\mathrm{Ch}}=\frac{s}{2}+\frac{1}{4}|dF|^2-\frac{1}{2}\delta^g \alpha_F
\end{equation}
and
\begin{equation}
S_2^{\mathrm{Ch}}=\frac{s}{2}+\frac{1}{4}|dF|^2-\frac{1}{2}|\alpha_F|^2-\delta^g \alpha_F.
\end{equation}
Then combining (5.4) and (5.5), (5.3) implies
\begin{equation}
\int_M (\hat{S}_1^{\mathrm{Ch}}+\hat{S}_2^{\mathrm{Ch}}) dV_{\hat{g}}=\int_M  e^{-2f}\left(s+\frac{1}{2}|dF|^2-\frac{1}{2}|\alpha_F|^2-\frac{3}{2}\delta^g \alpha_F \right)dV_{\hat{g}}.
\end{equation}

Using the formula (2.7), the above identity can be rewritten as
\begin{equation}
\int_M (\hat{S}_1^{\mathrm{Ch}}+\hat{S}_2^{\mathrm{Ch}}) dV_{\hat{g}}=\int_M  e^{-2f}\left(\frac{s+3s_J}{4}+\frac{1}{2}|dF|^2+\frac{1}{4}|\alpha_F|^2\right)dV_{\hat{g}}.
\end{equation}
Thus from (5.1), (5.2) and (5.7), we have $\int_M \hat{S}_1^{\mathrm{Ch}}dV_{\hat{g}}\geq 0$ if $3s_J+s\geq 0$.

If $\int_M \hat{S}_1^{\mathrm{Ch}}dV_{\hat{g}}>0$, then, as in the proof of Theorem 4.1,
we show that the Kodaira dimension satisfies $\kappa(M, J)=-\infty$.

If $\int_M \hat{S}_1^{\mathrm{Ch}}dV_{\hat{g}}=0$, then the integrals (5.2) and (5.7) vanish, which imply
 $dF=\alpha_F=0$, and the Riemannian scalar curvature $s=s_J=0$. Equivalently, $(M,J,g)$ is a K\"{a}hler manifold with zero Riemannian scalar curvature. Then, as in the proof of Theorem 4.1, the Kodaira dimension
$\kappa(M, J)=0$, and $(M,J,g)$ is a K\"{a}hler Calabi-Yau manifold.
\end{proof}

\begin{cor}
Let $(M,J)$ be a compact Moishezon manifold of complex dimension $n$. If $(M,J)$ admits a Hermitian metric $g$ such that its corresponding Riemannian scalar curvature $s$ and $J$-scalar curvature $s_J$ satisfing $3s_J+s>0$, then  $(M,J)$ is uniruled.
\end{cor}

\begin{proof}
If $3s_J+s>0$, then it follows from the proof of Theorem 5.1 that for the unique Gauduchon metric $\hat{g}$ with unit volume in the conformal class $[g]$, the corresponding total Chern scalar curvature is positive. The result is then a consequence of the characterization theorem (Theorem D) by Chiose-R\u{a}sdeaconu-\c{S}uvaina in \cite{CRS}.
\end{proof}

Let us recall the definition of the Riemannian holomorphic sectional curvature. At a point $p\in M$, for any nonzero $(1,0)$-type tangent vector $\xi=\sum_{i=1}^{n}\xi^i u_i\in T_p^{1,0}M$, the Riemannian holomorphic sectional curvature $H$ (with respect to the Levi-Civita connection $\nabla$) in the direction of $\xi$ is defined as
\begin{equation*}
H_p(\xi)=\frac{R(\bar{\xi},\xi,\xi,\bar{\xi})}{|\xi|^4}=\frac{\langle R(\xi,\bar{\xi})\xi,\bar{\xi}\rangle}{|\xi|^4}.
\end{equation*}
The Riemannian holomorphic sectional curvature $H$ is called nonnegative if $H_p(\xi)\geq 0$ for any pair $(p,\xi)$.
Set $S_p(M)=\{\xi\in T_p^{1,0}M |~|\xi|=1\}\cong S^{2n-1}$, $dV_{\xi}$ is the standard volume form for the sphere with radius 1.  Then we have
\begin{cor}
Let $(M,J,g)$ be a compact Hermitian manifold of complex dimension $n$. If its Riemannian holomorphic sectional curvature is nonnegative, then either

(1) $\kappa(M, J)=-\infty$; or

(2) $\kappa(M, J)=0$, and $(M,J,g)$ is a K\"{a}hler flat manifold.

\end{cor}

\begin{proof}
We need the following well-known Berger's averaging trick, which is used to show that the sign of Riemannian holomorphic sectional curvature
determines the sign of a mixed scalar curvature.
\begin{equation}
\begin{aligned}
\int_{S_p(M)} H_p(\xi) dV_{\xi}
&=\int_{S_p(M)} \sum_{i,j,k,l=1}^{n}R(u_{\bar{i}},u_j,u_k,u_{\bar{l}})\bar{\xi^{i}}\xi^{j}\xi^{k}\bar{\xi^{l}}dV_{\xi}\\
&=\sum_{i,j,k,l=1}^{n}R(u_{\bar{i}},u_j,u_k,u_{\bar{l}}) \int_{S_p(M)} \bar{\xi^{i}}\xi^{j}\xi^{k}\bar{\xi^{l}}dV_{\xi}\\
&=\sum_{i,j,k,l=1}^{n}R(u_{\bar{i}},u_j,u_k,u_{\bar{l}})\frac{\delta_{ij}\delta_{kl}+\delta_{ik}\delta_{jl}}{n(n+1)} \mathrm{Vol}(S^{2n-1})\\
&=\frac{\mathrm{Vol}(S^{2n-1})}{n(n+1)}\sum_{i,j=1}^{n}\bigl(R(u_{\bar{i}},u_i,u_j,u_{\bar{j}})+R(u_{\bar{i}},u_j,u_i,u_{\bar{j}})\bigr)\\
&=\frac{\mathrm{Vol}(S^{2n-1})}{4n(n+1)}(s+3s_J).
\end{aligned}
\end{equation}
If the Riemannian holomorphic sectional curvature is nonnegative, then the above equality implies  $s+3s_J\geq 0$.  From Theorem 5.1, we obtain that
(1) $\kappa(M, J)=-\infty$; or (2) $\kappa(M, J)=0$, and $(M,J,g)$ is a K\"{a}hler Calabi-Yau manifold. Furthermore, the holomorphic sectional curvature is zero in the second case. It is well-known that the holomorphic sectional curvature determines the sectional curvature
on a K\"{a}hler manifold. Therefore $(M,J,g)$ is a K\"{a}hler flat manifold in the second case. This completes the proof.

\end{proof}

\begin{rem}
Corollary 5.3 can be compared with the main result of Li \cite[Theorem 1.1]{Li1},  where he considered nonnegative holomorphic sectional curvature with respect to the Chern connection. Our condition is stronger than Li's. Indeed, for a Hermitian manifold, the holomorphic sectional curvature of the Levi-Civita connection is bounded above by that of the Chern connection \cite[Corollary 3.1]{Yu} \cite[Theorem 7]{YZ}. Nonnegative Riemannian holomorphic sectional curvature is a special case of the notion of mixed curvature, which was recently introduced by Chu-Lee-Zhu \cite[Definition 2.1]{CLZ}. It would be interesting to study the structure of compact Hermitian manifolds with nonnegative mixed curvature with respect to the Levi-Civita connection.
\end{rem}
\begin{rem}
On the other hand, in his  ``Problem section'', Yau~\cite[Problem 47]{Yau3} proposed a well-known conjecture
that a compact K\"{a}hler manifold with positive holomorphic sectional curvature is a projective and rationally connected manifold.
By studying the maximally rationally connected fibration, Heier-Wong \cite{HeW2} obtained that a projective K\"{a}hler manifold
with quasi-positive holomorphic sectional curvature is rationally connected. Yau's conjecture was  affirmatively confirmed by Yang \cite{Yang1}. Since then, there are many important generalizations and structure theorems for  compact K\"{a}hler manifold with
some nonnegative curvatures. We refer interested readers to \cite{CLT, CLZ, Li2, Mat1, Mat2, Ni, NZ1, NZ2, Tang, Zhang1, Zhang2} and references therein.
\end{rem}
\begin{rem}
We study (almost) complex geometry, particularly  the Kodaira dimension, from the perspective of Riemannian geometry. The results obtained in this work may be applied to construct certain examples, as suggested in Problem 5.4 of Chen-Zhang \cite{ChZ2}.
\end{rem}

\section{Appendix: Scalar curvatures on twistor spaces and some consequences}
Let $(N,g_N)$ be a compact oriented 4-dimensional Riemannian manifold with metric $g_N$. The Hodge star operator
gives a map $\ast_{g_N}: \wedge^2\rightarrow \wedge^2$ with $\ast_{g_N}^2=1$. Accordingly, its eigenvalues are $\pm 1$ and
the bundle of 2-forms splits $\wedge^2=\wedge^+\oplus \wedge^-$  into  eigenspaces.  $\wedge^+$ (resp. $\wedge^-$)
is called the bundle of self-dual (resp. anti-self-dual) 2-forms.

Set $\mathbf{Z}=\{(x,J_{x})|x\in N, J_{x}$ is an
orientation preserving orthogonal complex structure of the vector space $T_{x}N\}$.
$\mathbf{Z}$ is called the twistor space of $(N,g_N)$. $\mathbf{Z}$ has the following representations:
\begin{equation}
\mathbf{Z}=S(\wedge^{+})=P\times_{SO(4)}SO(4)/U(2),
\end{equation}
where $S(\wedge^{+})$ is the sphere bundle associated to the bundle of self-dual 2-forms $\wedge^{+}$, and  $P$ is the $SO(4)$-principal
bundle of oriented orthonormal frames over $(N,g_N)$. The Levi-Civita connection of $g_N$ gives rise to a splitting of the tangent bundle
$T\mathbf{Z}$ into horizontal and vertical components. Atiyah-Hitchin-Singer \cite{AHS} defined an important almost complex structure $\mathbb{J}_+$ on $\mathbf{Z}$, and they proved that $\mathbb{J}_+$ is integrable if and only if $(N,g_N)$ is anti-self-dual. On the other hand, Eells-Salamon \cite{ES} introduced another almost complex structure $\mathbb{J}_-$ which, by contrast with \textbf{$\mathbb{J}_+$}, is never integrable.
The twistor space $\mathbf{Z}$ also admits a  natural 1-parameter family of Riemannian metrics $g_t$ such that the projection
$\pi: \mathbf{Z}\rightarrow N$ is a Riemannian subersion with totally geodesic fibers. Furthermore, the metrics $g_t$ are compatible with almost complex structures $\mathbb{J}_+$ and $\mathbb{J}_-$. In the following, we employ the method introduced by
Jensen-Rigoli \cite{JR} \cite{FZ} to study some special geometric structures on the twistor space $(\mathbf{Z}, \mathbb{J}_{\pm})$.

On the $SO(4)$-principal bundle $P$ of oriented orthonormal frames over $(N,g_N)$, the $\mathbb{R}^4$-valued canonical form is denoted by $\theta=(\theta^a)$. The $\mathfrak{so}(4)$-valued Levi-Civita connection forms and curvature forms are denoted by $\omega=(\omega_b^{a})$ and $\Omega=(\Omega_b^{a})$, respectively. The structure equations of $(N,g_N)$ are
\begin{eqnarray}
&&d\theta^a=-\sum_{b=1}^{4}\omega_{b}^{a}\wedge\theta^b,\\
&&d\omega_b^{a}=-\sum_{c=1}^{4}\omega_c^{a}\wedge\omega_b^{c}+\Omega_b^{a},
\end{eqnarray}
where $\Omega_b^{a}=\frac{1}{2}\sum_{c,d=1}^{4}R_{abcd}\theta^c\wedge\theta^d$, $R_{abcd}$ are functions on $P$ defining the Riemannian curvature tensor, $a,b,c,d=1,2,3,4$.

We define complex-valued 1-forms on $P$ by
\begin{equation}
\varphi^1=\theta^1+\sqrt{-1}\theta^2,~\varphi^2=\theta^3+\sqrt{-1}\theta^4,~\varphi^3
=\theta^5+\sqrt{-1}\theta^6,
\end{equation}
where $\theta^5=\frac{1}{2}(\omega_3^1-\omega_4^2),\theta^6=\frac{1}{2}(\omega_4^1+\omega_3^2).$

As in \cite{JR}, let $U\subset \mathbf{Z}$ be an open subset on which there is a local section $\mathfrak{u}:U\rightarrow P$. Then an orthonormal coframe for $g_t$ on $U$ is given by
\begin{equation}
\mathfrak{u}^{\star}(\theta^a),~2t\mathfrak{u}^{\star}(\theta^5),~2t\mathfrak{u}^{\star}(\theta^6).
\end{equation}
The local complex-valued 1-forms $\{\mathfrak{u}^{\star}\varphi^1$,  $\mathfrak{u}^{\star}\varphi^2$,  $\mathfrak{u}^{\star}\varphi^3\}$ are the basis of (1,0)-forms of the almost complex structure $\mathbb{J}_+$, and the local complex-valued 1-forms $\{\mathfrak{u}^{\star}\varphi^1$,  $\mathfrak{u}^{\star}\varphi^2$,  $\mathfrak{u}^{\star}\overline{\varphi^3}\}$ are the basis of (1,0)-forms of the almost complex structure $\mathbb{J}_-$.

For the almost Hermitian twistor spaces $(\mathbf{Z},\mathbb{J}_{\pm},g_t)$, the associated fundamental 2-forms, denoted
by $F_{\pm}(t)$, are
\begin{equation}
F_{\pm}(t)=\frac{\sqrt{-1}}{2}(\mathfrak{u}^{\star}\varphi^1\wedge \mathfrak{u}^{\star}\overline{\varphi^1}+
\mathfrak{u}^{\star}\varphi^2\wedge \mathfrak{u}^{\star}\overline{\varphi^2} \pm 4t^2\mathfrak{u}^{\star}\varphi^3\wedge \mathfrak{u}^{\star}\overline{\varphi^3}),
\end{equation}
where $\mathfrak{u}$ is any local section of $\pi_1: P\rightarrow \mathbf{Z}$. Hereafter, for convenience, we always omit the pullback mapping $\mathfrak{u}^{\star}$.

With respect to the orthonormal coframe (6.5), the Levi-Civita connection forms for $(\mathbf{Z}, g_t)$ are given by \cite[(4.19)]{JR}
\begin{eqnarray}
&&\theta_b^{a}=\omega_{b}^{a}+t^2(R_{13ba}+R_{42ba})\theta^5+t^2(R_{14ba}+R_{23ba})\theta^6,\\
&&\theta_b^{5}=\frac{t}{2}(R_{13ba}+R_{42ba})\theta^a=-\theta_5^{b},\\
&&\theta_b^{6}=\frac{t}{2}(R_{14ba}+R_{23ba})\theta^a=-\theta_6^{b},\\
&&\theta_6^{5}=\omega_2^1+\omega_4^3=-\theta_5^{6}.
\end{eqnarray}
Using the structure equations, one can get all components of the Riemannian curvature tensor of $g_t$ \cite[(4.20)]{JR}. In particular, if $(N,g_N)$ is anti-self-dual and Einstein, then the Riemannian scalar curvature $s(t)$ of twistor space $(\mathbf{Z}, g_t)$ is
\begin{equation}
s(t)=s_N+\frac{2}{t^2}-\frac{s_N^2}{72}t^2,
\end{equation}
where $s_N$ is the Riemannian scalar curvature of $N$. Meanwhile, two $\mathbb{J}_{\pm}$-scalar curvatures are
\begin{equation}
s_{\mathbb{J}_+}(t)=s(t)=s_N+\frac{2}{t^2}-\frac{s_N^2}{72}t^2
\end{equation}
and
\begin{equation}
s_{\mathbb{J}_-}(t)=-\frac{1}{3}s_N+\frac{2}{t^2}+\frac{s_N^2}{24}t^2.
\end{equation}

It is also natural to consider the geometry of the Gauduchon connection \cite{Gau4} (specifically the Chern connection) on twistor spaces $(\mathbf{Z},\mathbb{J}_{\pm},g_t)$. Such problems were first studied by Davidov-Grantcharov-Mu\v{s}karov \cite{DGM}. In contrast, our discussion focuses on the scalar curvatures and complex geometry properties of these twistor spaces.

For the fundamental 2-forms $F_+(t)$  and $F_-(t)$ on the twistor space $\mathbf{Z}$, if the base manifold $(N,g_N)$ is anti-self-dual, then  $F_{\pm}(t)$ satisfy \cite{Mic}\cite{Mus}
\begin{equation}
F_{\pm}(t)\wedge dF_{\pm}(t)=0,
\end{equation}
i.e., the corresponding Lee forms $\alpha_{F_{\pm}(t)}=0$. According to a remark by Gauduchon \cite[(2.7.6)]{Gau4}, to obtain the Chern scalar curvatures of $(\mathbf{Z},\mathbb{J}_{\pm},g_t)$, it suffices to perform the calculation using the Lichnerowicz connection. From now on, we assume $(N, g_N)$ is anti-self-dual, then the corresponding Riemannian curvature components $R_{abcd}$ satisfy
\begin{equation}
R_{1312}+R_{4212}+R_{1334}+R_{4234}=0 \quad \text{and} \quad R_{1412}+R_{2312}+R_{1434}+R_{2334}=0.
\end{equation}

Using Kobayashi's formalism \cite[Section 5]{Kob2}, the connection form of the Lichnerowicz connection is given by its $u(3)$-component of the Levi-Civita connection. In the following, we will explicitly describe the connection forms associated with the structures $\mathbb{J}_+$
and $\mathbb{J}_-$ in the complex setting, respectively.

For Hermitian twistor space $(\mathbf{Z},\mathbb{J}_+,g_t)$, from (6.7)-(6.10), locally, the Lichnerowicz connection forms $\phi_{+j}^{i}$ are given by
\begin{eqnarray}
&&\phi_{+1}^{1}=-\sqrt{-1}\theta_2^1,~~\phi_{+2}^{1}=\frac{1}{2}\left(\theta_3^1+\theta_4^2+\sqrt{-1}(\theta_3^2-\theta_4^1)\right),\\
&&\phi_{+2}^{2}=-\sqrt{-1}\theta_4^3,~~\phi_{+3}^{2}=\frac{1}{2}\left(\theta_5^3+\theta_6^4+\sqrt{-1}(\theta_5^4-\theta_6^3)\right),\\
&&\phi_{+3}^{3}=-\sqrt{-1}\theta_6^5,~~\phi_{+3}^{1}=\frac{1}{2}\left(\theta_5^1+\theta_6^2+\sqrt{-1}(\theta_5^2-\theta_6^1)\right).
\end{eqnarray}
Similarly, for almost Hermitian twistor space $(\mathbf{Z},\mathbb{J}_-,g_t)$, locally, the Lichnerowicz connection forms $\phi_{-j}^{i}$ are given by
\begin{eqnarray}
&&\phi_{-1}^{1}=-\sqrt{-1}\theta_2^1,~~\phi_{-2}^{1}=\frac{1}{2}\left(\theta_3^1+\theta_4^2+\sqrt{-1}(\theta_3^2-\theta_4^1)\right),\\
&&\phi_{-2}^{2}=-\sqrt{-1}\theta_4^3,~~\phi_{-3}^{2}=\frac{1}{2}\left(\theta_5^3-\theta_6^4+\sqrt{-1}(\theta_5^4+\theta_6^3)\right),\\
&&\phi_{-3}^{3}=+\sqrt{-1}\theta_6^5,~~\phi_{-3}^{1}=\frac{1}{2}\left(\theta_5^1-\theta_6^2+\sqrt{-1}(\theta_5^2+\theta_6^1)\right).
\end{eqnarray}
Other components are determined by $\phi_{\pm j}^i+\overline{\phi_{\pm i}^j}=0$, where $i,j=1,2,3$. Using the structure equations, one can derive all curvature information for the Lichnerowicz connection of $(\mathbf{Z},\mathbb{J}_{\pm},g_t)$. The following are some key observations.

$\mathbf{(I)}$ With respect to $(\mathbf{Z},\mathbb{J}_+)$, it is well-known that $\mathbb{J}_+$ is integrable \cite{AHS}, $(\mathbf{Z},\mathbb{J}_+)$ is uniruled, and hence the Kodaira dimension $\kappa(\mathbf{Z},\mathbb{J}_+)=-\infty$. Furthermore, the first Chern class $c_1(\mathbf{Z},\mathbb{J}_+)$ is represented by the following Chern-Ricci form
\begin{equation}
\begin{aligned}
\rho_+(t)
&=\sqrt{-1}d(\phi_{+1}^{1}+\phi_{+2}^{2}+\phi_{+3}^{3})\\
&=d(\theta_2^1+\theta_4^3+\theta_6^5)=2d(\omega_2^1+\omega_4^3)\\
&=2(\omega_3^1-\omega_4^2)\wedge(\omega_4^1+\omega_3^2)+2(\Omega_2^1+\Omega_4^3)\\
&=\frac{2}{t^2}(2t\theta^5)\wedge(2t\theta^6)+2(\Omega_2^1+\Omega_4^3).
\end{aligned}
\end{equation}
In particular, if $(N,g_N)$ is anti-self-dual and Einstein, then
\begin{equation}
\Omega_2^1+\Omega_4^3=\frac{\sqrt{-1}}{24}s_N (\varphi^1\wedge\overline{\varphi^1}+\varphi^2\wedge\overline{\varphi^2}).
\end{equation}
Thus the corresponding Chern scalar curvature $s_1^{\mathrm{Ch}}(\mathbb{J}_+, t)$  is given by
\begin{equation}
s_1^{\mathrm{Ch}}(\mathbb{J}_+, t)=\langle \rho_+(t), F_+(t)\rangle=\frac{s_N}{3}+\frac{2}{t^2}.
\end{equation}
Now, Combining equation $(6.11)$, we obtain:

If $s_N>0$, then the twistor space $(\mathbf{Z},\mathbb{J}_+)$ admits balanced metrics with positive, negative and zero Riemannian scalar
curvature, respectively. However their Chern scalar curvatures are all positive.

If $s_N=0$, then the twistor space $(\mathbf{Z},\mathbb{J}_+)$ admits a natural 1-parameter family of balanced metrics with positive Riemannian scalar curvature and positive Chern scalar curvature.

If $s_N<0$, then the twistor space $(\mathbf{Z},\mathbb{J}_+)$ admits balanced metrics with positive, negative and zero Riemannian scalar
curvature, respectively. Meanwhile, the same conclusions also hold for Chern scalar curvature, i.e., $(\mathbf{Z},\mathbb{J}_+)$ admits balanced metrics with positive, negative and zero Chern scalar curvature, respectively. From Yang \cite[Theorem 1.1]{Yang2}, then the canonical line bundle and the anti-canonical line bundle of $(\mathbf{Z},\mathbb{J}_+)$ are not pseudo-effective.

$\mathbf{(II)}$ With respect to $(\mathbf{Z},\mathbb{J}_-)$, it is well-known that $\mathbb{J}_-$ is not integrable \cite{ES}, and that its first Chern class vanishes, i.e., $c_1(\mathbf{Z},\mathbb{J}_-)=0$, as also follows from the following Chern-Ricci form
\begin{equation}
\rho_-(t)=\sqrt{-1}d(\phi_{-1}^{1}+\phi_{-2}^{2}+\phi_{-3}^{3})=d(\theta_2^1+\theta_4^3-\theta_6^5)=0.
\end{equation}
Meanwhile, the corresponding Chern scalar curvature
\begin{equation}
s_1^{\mathrm{Ch}}(\mathbb{J}_-, t)=\langle \rho_-(t), F_-(t)\rangle=0.
\end{equation}

Let $K_{\mathbb{J}_-}=\Lambda_{\mathbb{J}_-}^{3,0}\mathbf{Z}$ be the canonical line bundle of almost complex manifold $(\mathbf{Z},\mathbb{J}_-)$. We claim that the Kodaira dimension $\kappa(\mathbf{Z},\mathbb{J}_-)=0$.

As in Jensen-Rigoli \cite{JR}, let $\mathfrak{u}$ be a local section of $\pi_1: P\rightarrow \mathbf{Z}$, then the $(3,0)$-form $\mathfrak{u}^{\star}\varphi^1\wedge\mathfrak{u}^{\star}\varphi^2\wedge\mathfrak{u}^{\star}\overline{\varphi^3}$ is globally defined, i.e., it is independent of the choice of the local section $\mathfrak{u}$. Therefore $\mathfrak{u}^{\star}\varphi^1\wedge\mathfrak{u}^{\star}\varphi^2\wedge\mathfrak{u}^{\star}\overline{\varphi^3}$ is a nowhere-vanishing section of  $K_{\mathbb{J}_-}$. For convenience, we always omit the pullback mapping $\mathfrak{u}^{\star}$.

Using the structure equations, a direct calculation shows that
\begin{equation}
\begin{aligned}
d(\varphi^1 \wedge \varphi^2 \wedge \overline{\varphi^3})
&= -\varphi^2 \wedge \overline{\varphi^2} \wedge \overline{\varphi^3} \wedge \varphi^3
   -\varphi^1 \wedge \overline{\varphi^1} \wedge \overline{\varphi^3} \wedge \varphi^3 \\
&\quad + \varphi^1 \wedge \varphi^2 \wedge \frac{1}{2}
   \left( \Omega_3^1 - \Omega_4^2 - \sqrt{-1} (\Omega_3^2 + \Omega_4^1) \right) \\
&= -(\varphi^1 \wedge \overline{\varphi^1}+\varphi^2 \wedge \overline{\varphi^2})\wedge \overline{\varphi^3} \wedge \varphi^3
-\frac{s_N}{24}\varphi^1 \wedge \overline{\varphi^1}\wedge \varphi^2 \wedge \overline{\varphi^2}.
\end{aligned}
\end{equation}
Thus
\begin{equation}
\bar{\partial}_{\mathbb{J}_-}(\varphi^1 \wedge \varphi^2 \wedge \overline{\varphi^3})=0.
\end{equation}
For any pseudoholomorphic section $\sigma\in H^0(\mathbf{Z}, K_{\mathbb{J}_-})$, then $\sigma=f \varphi^1 \wedge \varphi^2 \wedge \overline{\varphi^3}$, where $f$ is a smooth function on $\mathbf{Z}$. From $\bar{\partial}_{\mathbb{J}_-} \sigma=0$, we get $\bar{\partial}_{\mathbb{J}_-} f=0$. By the compactness of $\mathbf{Z}$ and the maximum principle, we deduce that $f$ must be a constant. Therefore $P_1(\mathbf{Z},\mathbb{J}_-)=1$ with $\varphi^1 \wedge \varphi^2 \wedge \overline{\varphi^3}$ being a generator. Similarly, we obtain $P_m(\mathbf{Z},\mathbb{J}_-)=1$ with $(\varphi^1 \wedge \varphi^2 \wedge \overline{\varphi^3})^{\otimes m}$ being a generator of $H^0(\mathbf{Z}, K_{\mathbb{J}_-}^{\otimes m})$. Thus the Kodaira dimension $\kappa(\mathbf{Z},\mathbb{J}_-)=0$.

\bigskip
\footnotesize
\noindent\textit{Acknowledgments.}
The author would like to thank Professors Jixiang Fu, Xiaoxiang Jiao and Jiagui Peng for their helpful suggestions and encouragements. He is also grateful to Wubin Zhou for inspiring discussions. This work was partially supported by NSFC 11501505.


\end{document}